\documentclass[a4paper,10pt]{amsart}
\usepackage[top=80pt,bottom=65pt,left=70pt,right=70pt]{geometry}

\usepackage{marginnote}

\usepackage{graphicx, eepic}
\usepackage{times}
\normalfont
\usepackage{bbm}
\usepackage{xcolor}
\usepackage{verbatim}
\usepackage{amscd}
\usepackage{float}
\usepackage{multirow,bigdelim}
\usepackage[subnum]{cases}
\usepackage{stmaryrd}
\usepackage{tikz}
\usetikzlibrary{arrows,matrix,positioning}

\usepackage{amsthm,amsmath,amsfonts,amssymb}
\usepackage{hyperref}
\hypersetup{colorlinks}
\hypersetup{
    bookmarks=true,         
    unicode=false,          
    pdftoolbar=true,        
    pdfmenubar=true,        
    pdffitwindow=false,     
    pdfstartview={FitH},    
    pdftitle={My title},    
    pdfauthor={Author},     
    pdfsubject={Subject},   
    pdfcreator={Creator},   
    pdfproducer={Producer}, 
    pdfkeywords={keyword1} {key2} {key3}, 
    pdfnewwindow=true,      
    colorlinks=true,       
    linkcolor=blue,          
    citecolor=red,        
    filecolor=violet,      
    urlcolor=violet       
}
\usepackage{multirow, url}
\usepackage{color}
\usepackage[all]{xy}
\theoremstyle{plain}
\newtheorem{theorem}{Theorem}[section]

\newtheorem{lemma}[theorem]{Lemma}
\newtheorem{corollary}[theorem]{Corollary}

\theoremstyle{definition}

\newtheorem{definition}[theorem]{Definition}
\newtheorem{example}[theorem]{Example}

\theoremstyle{remark}
\newtheorem{remark}[theorem]{Remark}

\newtheorem{notation}[theorem]{Notation}

\renewcommand{\bar}{\overline}
\renewcommand{\tilde}{\widetilde}

\newcommand{\dge}{\rotatebox[origin=c]{45}{$\ge$}}
\newcommand{\uge}{\rotatebox[origin=c]{315}{$\ge$}}

\newcommand{\updots}{\hbox to1.65em{\rotatebox[origin=c]{45}{$\cdots$}}}
\newcommand{\dndots}{\hbox to1.65em{\rotatebox[origin=c]{315}{$\cdots$}}}

\newcommand{\lmd}[1]{\hbox to1.65em{$\hfill \lambda_{#1} \hfill$}}
\newcommand{\llmd}[2]{\hbox to1.65em{$ \lambda_{#2}^{(#1)}$}}

\newcommand{\diag}{\operatorname{diag}}

\newcommand{\lij}{\lambda_i^{(j)}}

\newcommand{\C}{\mathbb{C}}

\newcommand{\Q}{\mathbb{Q}}
\newcommand{\R}{\mathbb{R}}
\newcommand{\Z}{\mathbb{Z}}

\newcommand{\F}{\mathcal{F}}
\newcommand{\PP}{\mathcal{P}}
\newcommand{\D}{\mathcal{D}}

\newcommand{\FF}{\mathbf{F}}
\newcommand{\VV}{\mathbf{V}}

\newcommand{\n}{\mathbf{n}}
\newcommand{\kk}{\mathbf{k}}
\newcommand{\xx}{\mathbf{x}}
\newcommand{\yy}{\mathbf{y}}

\newcommand{\bw}{\mathbf{w}}

\newcommand\ee{\mathbf{e}}

\newcommand{\ur}{\raisebox{3pt}{$\to$}\!\raisebox{-2pt}{$\uparrow$}}

\numberwithin{equation}{section} \numberwithin{table}{section}

\begin{document}                                                                          

\title{On the $f$-vectors of Gelfand-Cetlin polytopes}

\author[B. An]{Byung Hee An}
\address{Center for Geometry and Physics, Institute for Basic Science (IBS), Pohang, Republic of Korea 37673}
\email{anbyhee@ibs.re.kr}

\author[Y. Cho]{Yunhyung Cho}
\address{Center for Geometry and Physics, Institute for Basic Science (IBS), Pohang, Republic of Korea 37673}
\email{yhcho@ibs.re.kr}

\author[J. S. Kim]{Jang Soo Kim}
\address{Department of Mathematics, Sungkyunkwan University, Suwon 440-746, South Korea}
\email{jangsookim@skku.edu}

\thanks{The first and the second authors were supported by IBS-R003-D1. The third author was partially supported by Basic Science Research Program through the
National Research Foundation of Korea (NRF) funded by the Ministry of Education (NRF-
2013R1A1A2061006).}

\date{\today}

\begin{abstract}
	A {\em Gelfand-Cetlin polytope}
	is a convex polytope obtained as an image of certain completely integrable system on a partial flag variety. 
	In this paper, we give an equivalent description of the face structure of a GC-polytope 
	in terms of so called {\em the face structure of a ladder diagram}. Using our description, we obtain a partial differential equation 
	whose solution is the exponential generating function of $f$-vectors of GC-polytopes. This solves the open problem (2) posed by Gusev, Kritchenko, and Timorin
	in \cite{GKT}.
\end{abstract}
\maketitle
\setcounter{tocdepth}{1} 
\tableofcontents

\section{Introduction and statement of results}
\label{secIntroduction}

Let us fix a positive integer $n$ and let $\n = (n_0, n_1, \cdots, n_r, n_{r+1})$ be a sequence of integers such that 
$
0 = n_0 < n_1 < n_2 < \cdots <n_r < n_{r+1} = n
$
for some $r > 0$. 
For a sequence $\lambda = (\lambda_1, \cdots, \lambda_n)$ of real numbers such that 
\[
	\lambda_1 = \cdots = \lambda_{n_1} > \lambda_{n_1 + 1} = \cdots = \lambda_{n_2} > \cdots > \lambda_{n_r +1} = 
	\cdots = \lambda_{n}, 
\]
the {\em Gelfand-Cetlin polytope}, or simply {\em the GC-polytope}, denoted by $\PP_\lambda$ is a convex polytope lying on $\R^{d}$ ($d = \frac{n(n-1)}{2}$)
consisting of points 
$(\lij)_{i,j} \in \R^{d}$ satisfying 
\[
\lambda_i^{(j+1)} \ge \lambda_{i}^{(j)}\ge\lambda_{i+1}^{(j)}, \quad 1 \leq i \leq n-1, \quad  1 \leq j \leq n - i
\]
where $\lambda_i^{(n-i+1)} := \lambda_i$ for all $i = 1,\cdots,n$.
Equivalently, $(\lij)_{i,j} \in \PP_\lambda$ if and only if it satisfies 
\begin{equation}
\begin{alignedat}{17}
  \lmd 1 &&&& \lmd 2 &&&& \lmd 3 && \cdots && \lmd {n-1} &&&& \lmd n  \\
  & \uge && \dge && \uge && \dge &&\dndots &&&& \uge && \dge & \\
  && \llmd {n-1}1 &&&& \llmd {n-2}2 &&&&&&&& \llmd{1}{n-1} && \\
  &&& \uge && \dge && \dndots &&&&&& \dge &&& \\
  &&&& \llmd {n-2}1 &&&&&&&& \llmd{1}{n-2} &&&& \\
  &&&&& \uge &&&&&& \dge &&&&& \\
  &&&&&& \dndots &&&& \updots &&&&&& \\
  &&&&&&& \uge && \dge &&&&&&& \\
  &&&&&&&& \llmd 11 &&&&&&&&& 
\end{alignedat}
\label{equation_GC-pattern}
\vspace{0.2cm}
\end{equation}
for $1 \leq i \leq n-1$ and $~ 1 \leq j \leq n - i $.

The theory of GC-polytopes has been studied from various aspects, such as the representation theory of $GL_n(\C)$ 
(\cite{GC}, \cite{GKT}, \cite{LMc}), and the geometry of Schubert varieties (\cite{Ki}, \cite{Ko}, \cite{KM}, \cite{KST}). 
In the context of toric geometry, 
GC-polytopes correspond to (very singular) projective toric varieties which can be regarded as toric degenerations of flag varieties. 
Thus to study of GC-polytopes in the sense of convex geometry is one of the natural way of understanding how to degenerate  
flag varieties to projective toric varieties, see \cite{HK}. 

However, the combinatorics of GC-polytopes seems to be not quite well-understood. 
Recently, Gusev, Kiritchenko, and Timorin \cite{GKT} studied the number of vertices of GC-polytopes. 
More precisely, they provided certain PDE system such that the solution is a power series with multi-variable $x = (x_1, \cdots, x_n)$ such that 
each coefficient of $x^I$, where $I$ is an multi-index, is the number of vertices of the GC-polytope corresponding to $I$, see Section 
\ref{ssecTheoremOfGusevKiritchenkoTimorin} for more details. 

This paper concerns the enumerative combinatorics on Gelfand-Cetlin polytopes, in particular a counting the number faces in each dimension. 
Also, we provides the answer for the open question posed in \cite[open problem (2) of page 968]{GKT}, see
Theorem \ref{theorem_global_generating_function} and  Remark \ref{remark_GKT}.

\subsection{Geometric aspects of GC-polytopes}
A GC-polytope is closely related to the geometry of a partial flag variety, see \cite{Ki}, \cite{Ko}, \cite{KM}, and \cite{NNU}. 
 Even though we do not use the theory of GC-polytopes on the algebraic nor geometric 
aspects in this paper, we briefly explain a connection between GC-polytopes and the geometry of partial flag varieties as we see below.

A {\em partial flag variety} $\F\ell(\n)$ is an example of a projective Fano variety defined by 
\[
	\F\ell(\n) = \{V_{\bullet} := 0 \subset V_1 \subset \cdots \subset V_r \subset \C^n ~|~ \dim_{\C} V_i = n_i \}.
\]
We can easily check that the linear $U(n)$-action on $\C^n$ induces a transitive $U(n)$-action on $\F\ell(\n)$ with the stabilizer isomorphic to 
$U(k_1) \times \cdots \times U(k_{r+1})$ where $k_i = n_i - n_{i-1}$ for $i=1,\cdots,r+1$. In other words, $\F\ell(\n)$ is diffeomorphic to a homogeneous space 
\[
	\F\ell(\n)  \cong U(n) / U(k_1) \times \cdots \times U(k_{r+1}).
\]

In the symplectic point of view, $\F\ell(\n)$ can be described as a co-adjoint orbit of $U(n)$ as follows. 
Let $U(n)$ be the set of $n \times n$ unitary matrices and let $\mathfrak{u}(n)$ be the Lie algebra of $U(n)$, which is 
the set of $n \times n$ skew-hermitian matrices. Then we may identify the dual vector space $\mathfrak{u}(n)^*$ 
with the set of $n \times n$ hermitian matrices $\mathcal{H} = i\mathfrak{u}(n)$ via the inner product 
\[
	\langle X, Y \rangle = \mathrm{tr}(XY)
\]
on $\mathcal{H}$ so that $\mathfrak{u}(n)^*$ with the co-adjoint $U(n)$-action is $U(n)$-equivariantly diffeomorphic to $\mathcal{H}$ 
with the conjugate action of $U(n)$, see \cite[page 51]{Au} for the detail.

Let $I_\lambda$ be the diagonal matrix $I_\lambda=\diag(\lambda_1,\dots,\lambda_n) \in \mathcal{H}$. 
Then the orbit of $I_\lambda$ for the conjugate $U(n)$-action, denoted by $\mathcal{O}_\lambda$, has a stabilizer isomorphic to 
$U(k_1) \times \cdots \times U(k_{r+1})$ and hence we get 
\[
	\mathcal{O}_\lambda \cong U(n) / U(k_1) \times \cdots \times U(k_{r+1}) \cong \F\ell(\n). 
\]
In particular, we have 
\[
\dim_{\R} \F\ell(\n) = n^2 - \sum_{i=1}^{r+1} k_i^2.
\]
Together with the Kirillov-Kostant-Souriau symplectic form $\omega_\lambda$ on the co-adjoint orbit $\mathcal{O}_\lambda$, we get a symplectic manifold
$(\mathcal{O}_\lambda, \omega_\lambda)$ diffeomophic to $\F\ell(\n)$. 
Then the GC-polytope $\PP_\lambda$ is equal to the image of the following map 
\[
	\begin{array}{ccccl}
		\Phi_\lambda & : & \F\ell(\n) & \rightarrow & \R^d\\
		                        &   &           X                & \mapsto  & (\lij(X))_{i,j}\\
	\end{array}
\]
where $\{ \lambda_i^{(j)} \}_{i + j = \ell \geq 2}$ are eigenvalues of $(\ell-1) \times (\ell-1)$ principal minor $X^{(\ell-1)}$ of $X \in \mathcal{H}$ satisfying 
\[
	\lambda_1^{(\ell-1)}(X) \geq \lambda_2^{(\ell-2)}(X) \geq \cdots  \geq \lambda_{\ell-1}^{(1)}(X)
\]
for each $\ell = 2,\cdots, n$. 
Guillemin and Sternberg \cite{GS} proved that the map $\Phi_\lambda$ is a completely integrable system on $(\mathcal{O}_\lambda, \omega_\lambda)$,
called a {\em Gelfand-Cetlin system}, see \cite{GS} for more details. 

\subsection{Ladder diagrams}
In this paper, we study a combinatorial structure on GC-polytopes. 
More precisely, we study the {\em face lattice} of $\PP_\lambda$, denoted by 
$\mathcal{F}(\PP_\lambda)$, which consists of all faces of $\PP_\lambda$ graded by their geometric dimensions, 
and is equipped with the order relation 
given by the relation of inclusion of faces of $\PP_\lambda$.

Our first aim is to describe the face lattice of a GC-polytope in terms of a {\em
ladder diagram}. 
To define a ladder diagram, we first define $Q^+$ to be the
  infinite directed graph  with vertex set 
  \[
  	V(Q^+) := \Z_{\ge0}\times\Z_{\ge0}, 
  \] such that $((i,j),(i',j'))$ is a directed edge if and only if $(i',j')=(i,j+1)$ or $(i',j')=(i+1,j)$.

%
%

\begin{definition}\label{definition_ladder_diagram}
	For a given positive integer $n$, let $\kk = (k_1, \cdots, k_{s})$ be a sequence of positive integers such that $\sum_{i=1}^s k_i = n$.  
	Let $n_i=\sum_{1 \leq j \leq i} k_j$ for $i = 1,\cdots, s$ with $n_0 = 0$ and let 
			\[
				T_\kk=\{ (n_0, n-n_0), (n_1, n-n_1),\dots,(n_{s},n-n_{s})\} \subset V(Q^+). 
			\]
	\begin{enumerate}
		\item 
{\em The ladder diagram} $\Gamma_\kk$ is defined as the induced subgraph of $Q^+$ with vertex set
			\[
				V(\Gamma_\kk) = \{ (a,b) \in V(Q^+) ~|~ a\le c, ~b\le d \text{ for some } (c,d)\in T_\kk \}.
			\]
In other words, for two vertices $(a,b)$ and $(c,d)$ of $\Gamma_\kk$,
$((a,b),(c,d))$ is an edge of $\Gamma_\kk$ if and only if it is an
edge of $Q^+$. 
		\item $(0,0) \in V(Q^+)$ is called the {\em origin of $\Gamma_\n$}.
		\item A vertex $v \in T_\kk$ is called a {\em terminal vertex} of $\Gamma_\kk$. 
		\item A vertex $v \in V(\Gamma_\kk)$ is called {\em extremal} if $v$ is either a terminal vertex or the origin, and {\em non-extremal} otherwise.
	\end{enumerate}
\end{definition}

\begin{example}\label{example_def_ladder_diagram}
	The graphs $Q^+$, $\Gamma_{(1,1,1,1,1,1)}$, and $\Gamma_{(2,2,2)}$ are given as follows. 
	
\[
\begingroup%
  \makeatletter%
  \providecommand\color[2][]{%
    \errmessage{(Inkscape) Color is used for the text in Inkscape, but the package 'color.sty' is not loaded}%
    \renewcommand\color[2][]{}%
  }%
  \providecommand\transparent[1]{%
    \errmessage{(Inkscape) Transparency is used (non-zero) for the text in Inkscape, but the package 'transparent.sty' is not loaded}%
    \renewcommand\transparent[1]{}%
  }%
  \providecommand\rotatebox[2]{#2}%
  \ifx\svgwidth\undefined%
    \setlength{\unitlength}{347.84012451bp}%
    \ifx\svgscale\undefined%
      \relax%
    \else%
      \setlength{\unitlength}{\unitlength * \real{\svgscale}}%
    \fi%
  \else%
    \setlength{\unitlength}{\svgwidth}%
  \fi%
  \global\let\svgwidth\undefined%
  \global\let\svgscale\undefined%
  \makeatother%
  \begin{picture}(1,0.30212275)%
    \put(0,0){\includegraphics[width=\unitlength,page=1]{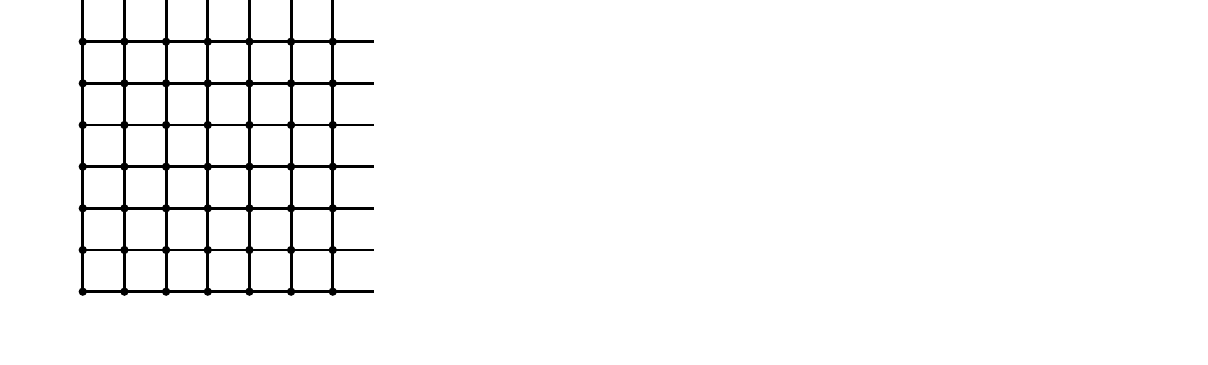}}%
    \put(0.0307534,0.03640139){\color[rgb]{0,0,0}\makebox(0,0)[lb]{\smash{$(0,0)$}}}%
    \put(0.16176599,0.00231352){\color[rgb]{0,0,0}\makebox(0,0)[lb]{\smash{$Q^+$}}}%
    \put(0,0){\includegraphics[width=\unitlength,page=2]{example_def_ladder_diagram.pdf}}%
    \put(0.25745853,0.03557992){\color[rgb]{0,0,0}\makebox(0,0)[lb]{\smash{$(6,0)$}}}%
    \put(-0.00087033,0.2639279){\color[rgb]{0,0,0}\makebox(0,0)[lb]{\smash{$(0,6)$}}}%
    \put(0.51701949,0.00313486){\color[rgb]{0,0,0}\makebox(0,0)[lb]{\smash{}}}%
    \put(0.46814653,0.00231338){\color[rgb]{0,0,0}\makebox(0,0)[lb]{\smash{$\Gamma_{(1,1,1,1,1,1)}$}}}%
    \put(0,0){\includegraphics[width=\unitlength,page=3]{example_def_ladder_diagram.pdf}}%
    \put(0.84598834,0.00313486){\color[rgb]{0,0,0}\makebox(0,0)[lb]{\smash{$\Gamma_{(2,2,2)}$}}}%
    \put(0,0){\includegraphics[width=\unitlength,page=4]{example_def_ladder_diagram.pdf}}%
  \end{picture}%
\endgroup%

\]
The red dots denote the terminal vertices for each graph.
\end{example}

\begin{remark}\label{remark_def_non_negative}
	Note that we defined a ladder diagram $\Gamma_\kk$ for a sequence $\kk$ of positive integers. 
	However, this definition of $\Gamma_\kk$ can be naturally extended for all sequences of non-negative integers such that 
	\[
		\Gamma_\kk := \Gamma_{\underline{\kk}}
	\]
	 where $\underline{\kk}$ is the maximal subsequence of $\kk$ whose components are all positive. 
\end{remark}

\begin{definition}[Definition 2.2.2 in \cite{BCKV}]
	A {\em positive path} on $\Gamma_\kk$ is a shortest path from the origin to a terminal vertex of $\Gamma_\kk$.
\end{definition}

\begin{definition}[A face structure on $\Gamma_\kk$] \label{definition_face}
	Let $\gamma$ be a subgraph of $\Gamma_\kk$. 
	\begin{enumerate}
		\item $\gamma$ is called a {\em face} of $\Gamma_\kk$ if 
			\begin{itemize}
				\item $V(\gamma)$ contains all terminal vertices of $\Gamma_\kk$, and 
				\item $\gamma$ can be presented as a union of positive paths.
			\end{itemize}
		\item For two faces $\gamma$ and $\gamma'$ of $\Gamma_\kk$, we say that $\gamma$ is a {\em face} of $\gamma'$ if $\gamma\subset\gamma'$.
		\item A {\em dimension} of a face $\gamma$ is defined by $\dim \gamma := \mathrm{rank}~ H_1(\gamma)$ by regarding $\gamma$ as a 
		one-dimensional CW-complex. In other words, $\dim\gamma$ is the number of minimal cycles in $\gamma$.
	\end{enumerate}
We denote by  $\mathcal{F}(\Gamma_\kk)$ the set of all faces of $\Gamma_\kk$. Then the face relation defined in (2) makes  $\mathcal{F}(\Gamma_\kk)$ a poset. In fact  $\mathcal{F}(\Gamma_\kk)$ is a lattice, see Remark~\ref{remark_intersection_of_face_need_not_face}. We call  $\mathcal{F}(\Gamma_\kk)$ the {\em face lattice} of $\Gamma_\kk$. 

\begin{remark}
	Let $\gamma$ be a face of $\Gamma_\kk$ and let $v$ be a non-extremal vertex in $V(\gamma)$. 
	Figure \ref{figure_jangsoo_1} illustrates the impossible types of the set of edges in $\gamma$ incident to $v$.  
\end{remark}

\end{definition}

Note that $\Gamma_\kk$ itself is a face of $\Gamma_\kk$ of maximal dimension, and we have
	\[
	\dim \Gamma_\kk = \mathrm{rank}~ H_1(\Gamma_\kk)=\sum_{1\le i<j\le s} k_i k_j = \frac{1}{2} \left( n^2 - \sum_{i=1}^s k_i^2 \right).
	\]	
\begin{example}\label{example_face_of_123}
	Let $\kk = (1,1,1)$. Then we can classify all faces of $\Gamma_\kk$ as in Figure \ref{figure_example_face_of_123}.
	\begin{figure}
		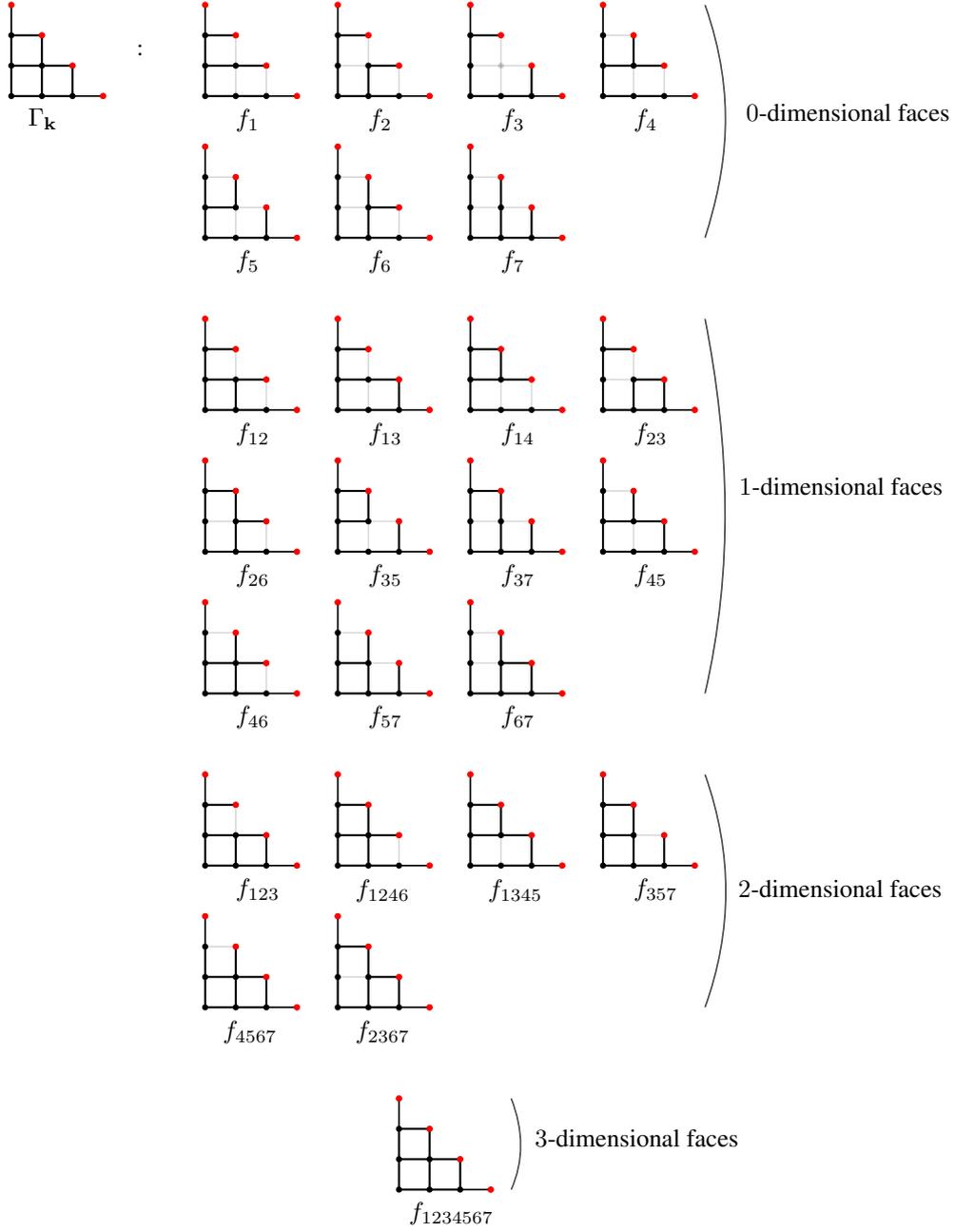
		\caption{\label{figure_example_face_of_123} The faces of $\Gamma_{(1,1,1)}$. }
	\end{figure}
	There are 7 faces of dimension zero, 11 faces of dimension one, 6 faces of dimension two, and 1 face of dimension three in $\Gamma_\kk$ as we see in Figure~\ref{figure_example_face_of_123}. 
	For faces $f_I$ and $f_J$ with $I, J \subset \{1,2,\cdots,7 \}$, 
	we can easily check that $f_I$ is a face of $f_J$ if and only if $I \subset J$. 
	In particular, we have $f_J = \cup_{j \in J} f_j$. 
\end{example}

\begin{remark}\label{remark_intersection_of_face_need_not_face}
	By definition, a union of faces of $\Gamma_\kk$ is again a face of $\Gamma_\kk$. In fact, if $\gamma_1, \cdots, \gamma_\ell$ are faces of $\Gamma_\kk$, 
	then $\cup_{i=1}^\ell \gamma_i$ is the smallest face containing all $\gamma_i$'s. Thus the union $\cup$ plays the role of the {\em join operator} $\vee$ for a lattice.
	On the other hand, the intersection of faces need not be a face. 
	For example,  $f_{123} \cap f_{357}$
	in Figure \ref{figure_example_face_of_123}
	cannot be expressed as a union of positive paths, and hence it is not a face of $\Gamma_\kk$ by Definition \ref{definition_face}.
	However, there is a unique maximal face $f_3$ contained in $f_{123}\cap f_{357}$. Thus 
	one can define the {\em meet} $\gamma \wedge \gamma'$ of two faces of $\Gamma_\kk$ as the maximal face contained in the intersection $\gamma \cap \gamma'$. 
	Then $\FF(\Gamma_\kk)$ becomes a lattice together with the join $\vee$ and the meet $\wedge$.
\end{remark}

%
The first part of our main theorem is as follows.
\begin{theorem}\label{theorem_one_to_one_faces}
	Let $\kk = (k_1, \cdots, k_2)$ be a sequence of positive integers and $\n = (n_0, \cdots, n_s)$ where $n_i = \sum_{j=1}^i k_j$ for $i=1,\cdots,s$ 
	with $n_0 = 0$. 
	Suppose that $\lambda = (\lambda_1, \dots, \lambda_n)$ is a sequence of real numbers satisfying 
	\[
		\lambda_1 = \cdots = \lambda_{n_1} > \cdots > \lambda_{n_{s-1} +1} = 
		\cdots = \lambda_{n_s}		
	\]
	Then there is an isomorphism $\phi$ between lattices 
\[
\phi : \mathcal{F}(\mathcal{P}_\lambda) \to \mathcal{F}(\Gamma_\kk)
\] 
such that $\dim\phi(F)=\dim F$ for all $F\in \mathcal{F}(\PP_\lambda)$. 
\end{theorem}

Note that Theorem \ref{theorem_one_to_one_faces} is equivalent to say that there exists a bijective map 
	\[
			\{~\text{faces of}~\Gamma_\kk \} \stackrel{\phi} \longrightarrow \{~\text{faces of} ~\PP_\lambda\}
	\]
such that 
	\begin{enumerate}
		\item $\dim \phi(F) = \dim F$, and
		\item $F \subset F' \Leftrightarrow \phi(F) \subset \phi(F')$
	\end{enumerate}
	for every faces $F$ and $F'$ of $\PP_\lambda$. In particular, $\phi$ preserves the operators $\vee$ and $\wedge$. 

\begin{example}\label{example_GC_polytope_123}
	Let $\lambda = (2,1,0)$.
	Then $\PP_\lambda$ is given as follows. 
	\begin{figure}[H]
		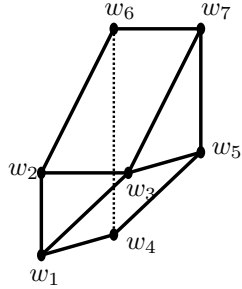
		\caption{\label{figure_GC_polytope_210}The GC-polytope $\PP_\lambda$ for $\lambda = (2,1,0).$}
	\end{figure}
	We label each vertex of $\PP_\lambda$ with $w_i$ for $i \in \{1,\cdots,7\}$ as given in Figure \ref{figure_GC_polytope_210}. 
	Similarly, we label each face of $\PP_\lambda$ with 
	$w_J$ for $J \subset \{1,\cdots,7\}$ such that $j \in J$ if and only if $w_J$ contains $w_j$. Then we can easily check that 
	\[
		\setlength\arraycolsep{2pt}
		\begin{array}{ccccc}
			\phi & : & \mathcal{F}(\Gamma_\kk) & \longrightarrow & \mathcal{F}(\PP_\lambda)\\
		                         & &      f_J         & \longmapsto & w_J
		\end{array}
	\]
	is an isomorphism where $f_J$ denotes a face of $\Gamma_\kk$ defined in \ref{example_face_of_123}.
\end{example}

\begin{remark}\label{remark_face_structure_depend_on_n}
	Note that Theorem \ref{theorem_one_to_one_faces} tells us that the face lattice
	 $\mathcal{F}(\PP_\lambda)$ of $\PP_\lambda$ depends only on $\kk$. 
\end{remark}

\subsection{Exponential generating functions of $f$-polynomials}
The second aim of this article is to study $f$-vectors of GC-polytopes by using Theorem \ref{theorem_one_to_one_faces}. 
%

Let $\kk$ be a sequence of non-negative integers and let $\Gamma_\kk$ be the corresponding ladder diagram in the sense of 
Remark \ref{remark_def_non_negative}.
Let $f_i(\kk)$ be the number of faces of $\Gamma_\kk$ of dimension $i$ for $i=0,1,\cdots, \dim \Gamma_\kk$.
We call $f(\kk) := (f_0(\kk), \cdots, f_{\dim \Gamma_\kk}(\kk))$ the {\em $f$-vector} of $\Gamma_\kk$. 
Then {\em the $f$-polynomial} $\FF_\kk(t)$ of $\Gamma_\kk$ is defined by 
\[
\FF_\kk(t) := \sum_{i=0}^{\dim \Gamma_\kk} f_i(\kk)  t^i =\sum_{\gamma \in \F(\Gamma_\kk)} t^{\dim \gamma},
\]
where $t$ is a formal parameter. 
In particular, the number of zero--dimensional faces 
of $\Gamma_\kk$, denoted by $\VV_\kk$, is equal to $\FF_\kk(0)$. 

For each positive integer $s$, we define the power series $\Psi_s$ 
in formal variables $x_1, \cdots, x_s,$ and $t$ as 
\begin{equation}
\Psi_0(t):=1,\qquad
\Psi_s(x_1,\dots,x_s ;t) := \sum_{k_1,\dots, k_s\ge0} \FF_{(k_1,\dots, k_s)}(t) \frac{x_1^{k_1}\cdots x_s^{k_s}}{k_1!\cdots k_s!}.
\end{equation}

For the sake of simplicity, we denote by 
\begin{equation}\label{equation_definition_global_generating_function}
	\Psi_s(\xx;t) = \sum_{\kk \in \Z^s_{\ge 0}} \FF_\kk(t) \frac{\xx^\kk}{\kk !}
\end{equation}
where $\xx = (x_1, \cdots, x_s)$,  
$\xx^\kk = x_1^{k_1}\cdots x_s^{k_s}$, and $\kk ! = k_1!\cdots k_s!$.
Then we can prove the following.
\begin{theorem}\label{theorem_global_generating_function}
The following equation
	\[
		\left( \D_s(\Psi_{2s-1}(\xx * \yy ; t))\right)|_{\yy = 0} = 0
	\]
	holds for every positive integer $s$
	where 
	\[
	\xx * \yy = (x_1, y_1, \cdots, x_{s-1}, y_{s-1}, x_s)
	\]
	for $\xx = (x_1, \cdots, x_s)$ and $\yy = (y_1, \cdots, y_{s-1})$, and
	\[
		\D_s = \displaystyle \frac{\partial^s}{\partial x_1 \cdots \partial x_s} - \prod_{i=1}^{s-1}\left (\frac{\partial}{\partial x_i} + \frac{\partial}{\partial x_{i+1}}
		+ t \cdot \frac{\partial}{\partial y_i}\right ).
	\]
\end{theorem}

\subsection{Theorem of Gusev-Kiritchenko-Timorin}
\label{ssecTheoremOfGusevKiritchenkoTimorin}
As a corollary of Theorem \ref{theorem_global_generating_function}, 
we obtain the following result proved by Gusev, Kiritchenko, and Timorin, see also \cite[Theorem 1.1]{GKT}. 

\begin{corollary}\cite{GKT}\label{corollary_GKT}
For $s\ge1$, let
	\[
		E_s(\xx) := \Psi_s(\xx ; 0) = \displaystyle \sum_{\kk \in \Z^s_{\ge 0}} \VV_\kk \frac{\xx^\kk}{\kk !}.
	\]
	where $\VV_\kk := \FF_\kk(0)$ is the number of vertices of $\Gamma_\kk$. 
Then $E_s(\xx)$ is a solution of the following partial differential equation
	\[
		\displaystyle \left(\frac{\partial^s}{\partial x_1 \cdots \partial x_s} - \prod_{i=1}^{s-1}\left (\frac{\partial}{\partial x_i} + \frac{\partial}{\partial x_{i+1}} \right ) 
		\right) E_s(\xx) = 0.
	\]
\end{corollary}

\begin{proof}
	For $s \ge 1$, let us denote by 
	\[
		\D'_s =\displaystyle \frac{\partial^s}{\partial x_1 \cdots \partial x_s} - \prod_{i=1}^{s-1}\left (\frac{\partial}{\partial x_i} + \frac{\partial}{\partial x_{i+1}} \right ) .
	\]
Then $\D_s = \D'_s + t \cdot \D''_s$ for some partial differential operator $\D''_s$. 
	Observe that 
	\begin{itemize}
		\item $\Psi_s(\xx ; 0) = E_s(\xx)$, 
		\item $\Psi_s(\xx ; t) = \Psi_{2s-1}(\xx * 0 ; t) = \Psi_{2s-1}(\xx * \yy ; t)|_{\yy = 0}$, and
		\item $\D'_s(\Psi_s(\xx ; t))|_{t = 0} = \D'_s(\Psi_s(\xx ; 0))$. 
	\end{itemize}
	Then by Theorem \ref{theorem_global_generating_function}, 
	we obtain
	\[
		\begin{array}{ccl}
			0 & = & \left( \D_s(\Psi_{2s-1}(\xx * \yy ; t))\right)|_{\yy = 0} \\
				& = & 	\left( \D'_s(\Psi_{2s-1}(\xx * \yy ; t))\right)|_{\yy = 0} + \left( 
			t \cdot \D''_s(\Psi_{2s-1}(\xx * \yy ; t))\right)|_{\yy = 0} \\
			& = & \D'_s(\Psi_s(\xx ; t)) + \left( 
			t \cdot \D''_s(\Psi_{2s-1}(\xx * \yy ; t))\right)|_{\yy = 0}\\
		\end{array}
	\]	
	for every $t \in \R$. Thus by substituting $t = 0$, we have 
	\[
		\D_s'(\Psi_s(\xx;t))|_{t = 0}  = \D'_s(\Psi_s(\xx ; 0)) = \D_s'(E_s(\xx)) = 0. 
	\]
	which completes the proof.
\end{proof}

\begin{remark}\label{remark_GKT}
Finding a partial differential equation whose solution is the exponential generating function of $f$-polynomials of GC-polytopes was an open problem posed by Gusev, Kiritchenko, and Timorin in \cite{GKT}.  
Thus Theorem \ref{theorem_global_generating_function} gives the answer for the problem.  
\end{remark}

This paper is organized as follows. In Section \ref{secProofOfTheoremRefTheoremonetoonefaces}, we give the proof of 
Theorem \ref{theorem_one_to_one_faces}. And in Section \ref{secProofOfTheoremRefTheoremglobalgeneratingfunction}, we give the proof of 
Theorem \ref{theorem_global_generating_function}. 

\section{Face lattices of ladder diagrams}
\label{secProofOfTheoremRefTheoremonetoonefaces}

In this section, we study face lattices of ladder diagrams defined in 
Section \ref{secIntroduction}
and prove Theorem \ref{theorem_one_to_one_faces}. 

Let us fix an integer $n>1$, a sequence
$\lambda = (\lambda_1, \dots, \lambda_n)$ of real numbers satisfying
	\[
		\lambda_1 = \cdots = \lambda_{n_1} > \lambda_{n_1 + 1} = \cdots = \lambda_{n_2} > \cdots > \lambda_{n_r +1} = 
		\cdots = \lambda_{n},
	\]
and a sequence $\kk = (k_1,\dots, k_{r+1})$ with $k_i=n_i-n_{i-1}$ for $i=1,2,\dots,r+1$, where $n_0 = 0$ and $n_{r+1} = n$. 

Let $I = \{(i,j) \in \Z^2 ~|~  i,j\ge 1, i+j\le n\}$ be an index set with $|I| = d := {n\choose 2}$. As in (\ref{equation_GC-pattern}), 
we denote the coordinates of $\R^d$ by $x_I = (x_{i,j})_{(i,j)\in I} \in \R^d$ so that 
$\PP_\lambda$ is written by 
\[
\PP_\lambda = \{x_{I} ~|~ x_{i,j+1}\ge x_{i,j}\ge
x_{i+1,j} \mbox{ for $(i,j)\in I$}\},
\]
where $x_{i,n+1-i} = \lambda_i$ for $i=1,2,\dots,n$. 

For each face $F\in \mathcal{F}(\PP_\lambda)$, let us define the subgraph $\phi(F)$ of $Q^+$ 
whose edge set is 
\begin{multline*}
E(\phi(F)) = \{((0,i),(0,i+1)) ~|~ 0\le i\le n-1\}  \cup 
\{((i,0),(i+1,0)) ~|~ 0\le i\le n-1\}  \\
\cup \{((i-1,j),(i,j)) ~|~ \mbox{ if there is a point $x_I \in F$ with $x_{i,j}<x_{i,j+1}$}\}\\
\cup \{((i,j-1),(i,j)) ~|~ \mbox{ if there is a point $x_I \in F$ with $x_{i,j}>x_{i+1,j}$}\},
\end{multline*}
and vertex set $V(\phi(F))$ is defined to be a subset of $V(Q^+)$ whose element is an endpoint of an edge in $E(\phi(F))$. 
See Figure \ref{figure_jangsoo_2} for an illustration of the possible sets of edges incident to the vertex
$(i,j) \in V(\Gamma_\kk)$ and the coordinates $x_{i,j}$'s for each $(i,j) \in I$ (cf. (\ref{equation_GC-pattern})).

\begin{figure}
	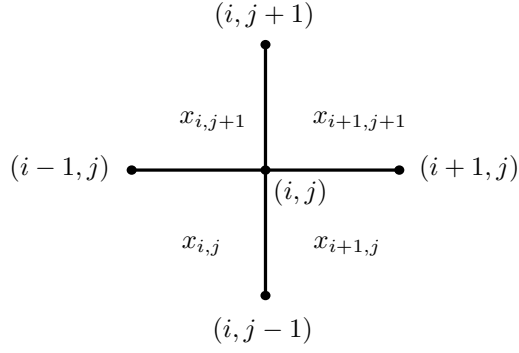
	\caption{\label{figure_jangsoo_2} Coordinates $x_I$ of $\R^d$}
\end{figure}

\begin{lemma}\label{lemma_subgraph_of_kk}
	$\phi(F)$ is a subgraph of $\Gamma_\kk$.
\end{lemma}

\begin{proof}
	It is enough to show that each edge of $E(\phi(F))$ is lying on $\Gamma_\kk$.
	Let $e = ((i-1,j), (i, j))$ be any horizontal edge in $Q^+$ not lying on $\Gamma_\kk$. Then by Definition \ref{definition_ladder_diagram}, 
	there is no terminal vertex $(c,d) \in T_\kk$ such that $ i \leq c$ and $j \leq d$. Equivalently, $(i,j) \not \in V(\Gamma_\kk)$ so that
	there exist consecutive terminal vertices $(n_\ell, n- n_\ell)$ and $(n_{\ell+1}, n - n_{\ell + 1})$ for some $0 \leq \ell \leq r$ such that 
	\[
		n_\ell < i <n_{\ell+1}, \quad \mathrm{and} \quad n - n_{\ell+1} < j <n - n_\ell.
	\]
	Then we have $x_{i+1,j} = x_{i,j} =x_{i,j+1} = \lambda_{n_\ell+1}$ by (\ref{equation_GC-pattern}) and hence 
	$e$ cannot be lying on $\phi(F)$ by definition of $\phi$, i.e., any edge of $\phi(F)$ is lying on $\Gamma_\kk$ 
	for any $F \in \F(\PP_\lambda)$. Similarly, we can easily see that the same argument holds for a vertical edge of $Q^+$ so that  
	$\phi(F)$ is a subgraph of $\Gamma_\kk$ for every $F \in \F(\PP_\lambda)$.
\end{proof}

\begin{lemma}\label{lemma_contain_top_origin}
	$\phi(F)$ contains every terminal vertex and the origin of $\Gamma_\kk$.
\end{lemma}

\begin{proof}
	It is clear that $\phi(F)$ contains the origin and two terminal vertices $(0,n)$ and $(n,0)$ by definition of $E(\phi(F))$ and $V(\phi(F))$. 
	Now, let us suppose that a terminal vertex $(n_\ell, n-n_\ell)$ of $\Gamma_\kk$ is not contained in $\phi(F)$ for some $1 \leq \ell \leq r$.
	Then we can see that two edges $((n_\ell - 1, n- n_\ell), (n_\ell, n- n_\ell))$ and $((n_\ell, n- n_\ell-1), (n_\ell, n- n_\ell))$ are not in $E(\phi(F))$
	which implies that any point $x_I \in F$ satisfies $\lambda_{N_\ell} = x_{n_\ell, n-n_\ell} = \lambda_{n_{\ell}+1}$, and this contradicts 
	$\lambda_{n_\ell} > \lambda_{n_\ell+1}$.
	Therefore, $\phi(F)$ contains every terminal vertices of $\Gamma_\kk$.
\end{proof}

\begin{lemma}\label{lemma_extend_positive_path}
	$\phi(F)$ does not contain a non-extremal\footnote{See Definition \ref{definition_ladder_diagram} (4).} vertex $(i,j)$ which is one of six types in Figure \ref{figure_jangsoo_1}. Equivalently, every edge $e \in E(\phi(F))$ can be extended to a positive path lying on $\phi(F)$.
\end{lemma}

\begin{proof}
	It is straightforward by (\ref{equation_GC-pattern}). 
\end{proof}

\begin{figure}
	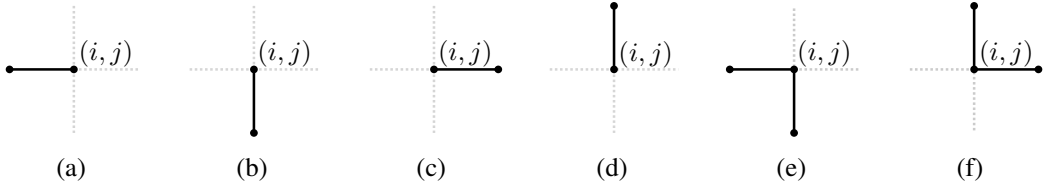
	\caption{\label{figure_jangsoo_1} The impossible sets of edges in a face of $\Gamma_\kk$ incident to 
	non-extremal vertex $(i,j) \in V(\Gamma_\kk)$.} 
\end{figure}

%

By Lemma \ref{lemma_subgraph_of_kk}, Lemma \ref{lemma_contain_top_origin}, and Lemma \ref{lemma_extend_positive_path}, we have the following
corollary.

\begin{corollary}\label{corollary_face}
For any $F\in \mathcal{F}(\PP_\lambda)$, we have
  $\phi(F)\in \mathcal{F}(\Gamma_\kk)$.
\end{corollary}
%
%

\begin{lemma}\label{lemma_bijection}
The map $\phi:\mathcal{F}(\PP_\lambda)\to\mathcal{F}(\Gamma_\kk)$
is a bijection. 
\end{lemma}
\begin{proof}
We will show this by constructing an inverse of $\phi$. 
Let $\gamma \in \mathcal{F}(\Gamma_\kk)$. Then we define $\psi(\gamma)$ to be the
set of points $x_I\in \PP_\lambda$ such that
\begin{itemize}
\item if $((i-1,j),(i,j))\not\in \gamma$, then $x_{i,j} = x_{i,j+1}$, and
\item if $((i,j-1),(i,j))\not\in \gamma$, then $x_{i,j} = x_{i+1,j}$.
\end{itemize}
It is obvious that $\psi(\gamma )\in \mathcal{F}(\PP_\lambda)$ since $\psi(\gamma)$ is the intersection of $\PP_\lambda$ and some facet hyperplanes determined by 
the above equalities. 

Let $\gamma'=\phi(\psi(\gamma))$. 
In order to show that $\psi$ is the inverse of $\phi$, we need to show
$\gamma'=\gamma$. In fact, we only need to show $E(\gamma')=E(\gamma)$. 
From
the construction of $\psi(\gamma)$, every edge not in $\gamma$ is not in $\gamma'$, which implies $E(\gamma')\subseteq E(\gamma)$. 
Thus it remains to show
that $E(\gamma)\subseteq E(\gamma')$.

Let us consider the point $\bar{x}_I = (\bar{x}_{i,j})_{(i,j)\in I} \in \psi(\gamma)$ defined recursively as follows:
\begin{itemize}
\item Set $\bar{x}_{i,n+1-i} = \lambda_i$ for $i=1,2,\dots,n$. 
\item If $\bar{x}_{i,j+1}, \bar{x}_{i+1,j}$ are defined, then $\bar{x}_{i,j}$ is defined
  by
\[
\bar{x}_{i,j} = \left\{
\begin{array}{ll}
  \bar{x}_{i,j+1} & \mbox{ if } ((i-1,j),(i,j))\not\in E(\gamma),\\
  \bar{x}_{i+1,j} & \mbox{ if } ((i-1,j),(i,j))\in E(\gamma) \mbox{ and } ((i,j-1),(i,j))\not\in E(\gamma),\\
  \frac12(\bar{x}_{i,j+1}+\bar{x}_{i+1,j}) & \mbox{ if } ((i-1,j),(i,j))\in E(\gamma) \mbox{ and } ((i,j-1),(i,j))\in E(\gamma).
\end{array}
\right.
\]
\end{itemize}

\begin{figure}[H]
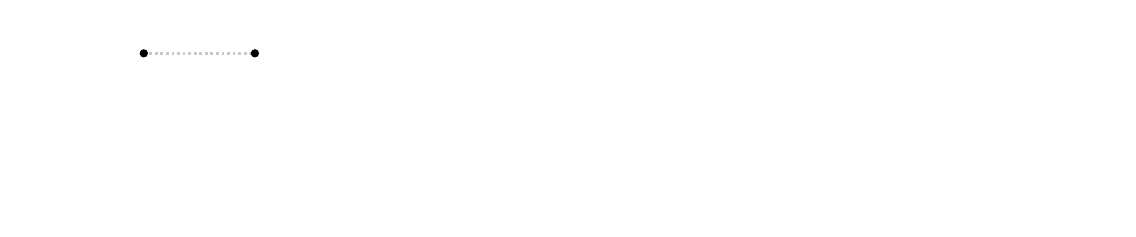
\end{figure}

Then we claim that 
\begin{description}
\item [C1] If $((i-1,i),(i,j))\in E(\gamma)$, then $\bar{x}_{i,j} < \bar{x}_{i,j+1}$, and
\item [C2] if $((i,j-1),(i,j))\in E(\gamma)$, then $\bar{x}_{i,j} > \bar{x}_{i+1,j}$, 
\end{description}
which implies that $E(\gamma)\subseteq E(\gamma')$, which will finish the proof.

For the proof of the claim, suppose that it is false. Then we can find
a lexicographically maximal\footnote{The lexicographic order on $V(\Gamma_\kk)$ is defined by 
$(i,j) \leq (i',j')$ if and only if $i \leq i'$, or $i=i'$ and $j \leq j'$.} vertex $(i,j)$ for which
\textbf{C1} or \textbf{C2} is false. Suppose that \textbf{C1} is
false. Then we have $((i-1,j),(i,j))\in E(\gamma)$ and $\bar{x}_{i,j}=\bar{x}_{i,j+1}$.
\vspace{0.2cm}

\textsl{CASE 1}: $((i,j-1),(i,j))\not\in \gamma$.  By definition of
$\bar{x}_I$ and by our assumption, we have $\bar{x}_{i,j}=\bar{x}_{i+1,j} = \bar{x}_{i,j+1}$. If $i+j=n$, then we have
$\lambda_i = \bar{x}_{i,j+1} = \bar{x}_{i,j} = \bar{x}_{i+1,j} = \lambda_{i+1}$. 
Then $\lambda_i = \lambda_{i+1}$
implies that  $((i-1,j),(i,j)) \not \in E(\Gamma_\kk)$ by definition of $\Gamma_\kk$, and hence $((i-1,j),(i,j)) \not \in E(\gamma)$ which contradicts 
the assumption that $((i-1,j),(i,j)) \in E(\gamma)$. Thus we must have
$i+j<n$. Since $\bar{x}_{i,j+1}\ge \bar{x}_{i+1,j+1} \ge \bar{x}_{i+1,j}$ and
$\bar{x}_{i,j+1} = \bar{x}_{i+1,j}$, we have $\bar{x}_{i,j+1}=\bar{x}_{i+1,j+1} =\bar{x}_{i+1,j}$.
Since we have taken $(i,j)$ to be a maximal vertex  (with respect to the lexicographic order) on which \textbf{C1} or \textbf{C2} fails, we have
$((i,j),(i,j+1))\not\in E(\gamma)$ and $((i,j),(i+1,j))\not\in E(\gamma)$. Then
$((i-1,j),(i,j))$ is the only edge incident to $(i,j)$ in $\gamma$, which
contradicts Lemma~\ref{lemma_extend_positive_path}.
\vspace{0.2cm}

\textsl{CASE 2}: $((i,j-1),(i,j))\in \gamma$.  In this case, we have $\bar{x}_{i,j}=\frac12(\bar{x}_{i,j+1}+\bar{x}_{i+1,j})$. Since
$\bar{x}_{i,j+1}\ge \bar{x}_{i,j} \ge \bar{x}_{i+1,j}$ and
$\bar{x}_{i,j}=\frac12(\bar{x}_{i,j+1}+\bar{x}_{i+1,j})=\bar{x}_{i,j+1}$ by our assumption, we have
$\bar{x}_{i,j+1}= \bar{x}_{i,j}= \bar{x}_{i+1,j}$. Thus we may apply the same argument as in \textsl{CASE 1}, and hence we can deduce a contradiction. 

Thus $(i,j)$ satisfies \textbf{C1}. Similarly we can show that $(i,j)$
also satisfies \textbf{C2}. Thus it completes the proof of our claim \textbf{C1} and \textbf{C2}.
\end{proof}

The following lemma completes the proof of Theorem \ref{theorem_one_to_one_faces}. 

\begin{lemma}\label{lem:3}
  The map $\phi:\mathcal{F}(\PP_\lambda)\to\mathcal{F}(\Gamma_\kk)$ is
  a poset isomorphism. Moreover, we have 
  \[
  	\dim F = \dim \phi(F)
  \]
  for every $F \in \F(\PP_\lambda)$.
\end{lemma}
\begin{proof}
  Let $F_1$ and $F_2$ be faces of $\PP_\lambda$ such that $F_1\subseteq F_2$. Then $F_1$ is an intersection of 
  $F_2$ and some facet hyperplanes, i.e., 
  $F_1$ is obtained from $F_2$ by replacing
  some inequalities $x_{i,j+1}\ge x_{i,j}$ or $x_{i+1,j}\ge x_{i,j}$
  by equalities $x_{i,j+1}= x_{i,j}$ or $x_{i+1,j}= x_{i,j}$.  By the
  definition of $\phi$, in this case $\phi(F_1)$ is obtained from
  $\phi(F_2)$ by removing corresponding edges. Thus we have
  $\phi(F_1)\subseteq \phi(F_2)$. Conversely, suppose that
  $\phi(F_1)\subseteq \phi(F_2)$. By the construction of the inverse
  map of $\phi$ in the proof of Lemma~\ref{lemma_bijection}, we clearly have
$F_1\subseteq F_2$. Thus $\phi$ is a poset isomorphism.

For the dimension formula, recall that  
\[
\dim \mathcal{P}_\lambda  = \dim \Gamma_\kk = \frac{1}{2} \left( n^2 - \sum_{i=1}^s k_i^2 \right).
\]
Since $\phi$ is a poset isomorphism, $\phi$ maps a maximal chain $F_0 \subset F_1 \subset \cdots \subset F_{\dim P_\lambda}$ in $\F(P_\lambda)$ 
to the maximal chain $\phi(F_0) \subset \phi(F_1) \subset \cdots \subset \phi(F_{\dim P_\lambda})$ in $\F(\Gamma_\kk)$. Then the dimension formula
follows from the simple observation that $\gamma  \subsetneq \gamma'$ implies $\dim \gamma < \dim \gamma'$ for every $\gamma, \gamma' \in \F(\Gamma_\kk)$.

%
\end{proof}

\section{Exponential generating functions and PDE systems}
\label{secProofOfTheoremRefTheoremglobalgeneratingfunction}

In this section we study the exponential generating function of $f$-polynomials of $\Gamma_\kk$'s defined by 
\[
		\Psi_s(\xx;t) = \sum_{\kk \in \Z^s_{\ge 0}} \FF_\kk(t) \frac{\xx^\kk}{\kk !}
\]
where $\FF_\kk(t)$ is the $f$-polynomial of $\Gamma_\kk$.
Also, we give the complete proof of Theorem \ref{theorem_global_generating_function}.
\subsection{Notations}
To begin with, we first introduce some notations as follows. 
\begin{notation} Let $s \geq 1$ be an integer. 
	\begin{itemize}
		\item For a multivariable $\xx = (x_1, \cdots, x_s)$ and $\mathbf{a} = (a_1, \cdots, a_s) \in \Z^s_{\ge 0}$, 
			\[
				\xx^{\mathbf{a}} := x_1^{a_1}\cdots x_s^{a_s}, \quad \mathbf{a}! := a_1!\cdots a_s!.
			\]	
		\item For another multivariable $\yy = (y_1, \cdots, y_{s-1})$, 
			\[
				\xx * \yy := (x_1, y_1, \cdots, x_{s-1}, y_{s-1}, x_s).
			\]
			In particular, we have 
			\[
				(\xx*\yy)^{\mathbf{a}*\mathbf{b}}=\xx^{\mathbf{a}} \yy^{\mathbf{b}},\qquad
				(\mathbf{a}*\mathbf{b})!=\mathbf{a}! \cdot \mathbf{b}!.
			\]
	\end{itemize}
\end{notation}
\begin{notation}
	Let $W_{s-1}$ be the set of all sequences of length $(s-1)$ on the set $\{1,0), (0,1), (1,1)\}$, i.e., each element of $W_{s-1}$ is of the form 
			\[
				\bw=((\alpha_1,\beta_1),\dots,(\alpha_{s-1},\beta_{s-1})), \quad (\alpha_i,\beta_i)\in \{(1,0),(0,1),(1,1)\} ~\text{for}~1\le i\le s-1.
			\]
	In particular, we have $\#(W_{s-1}) = 3^{s-1}$. 
	For $\kk = (k_1, \cdots, k_s) \in \Z^s$ and $\bw = ((\alpha_1, \beta_1), \cdots, (\alpha_{s-1}, \beta_{s-1})) \in W_{s-1}$, we denote by   
	\begin{itemize}
		\item $\alpha_s = \beta_0 = 1$, 
		\item $d_\bw(\kk) = (k_1',\dots,k_s') \in \Z^s$ where 
			\[
				k_i' = k_i-(1-\alpha_i)-(1-\beta_{i-1}), \qquad 1\le i\le s,
			\]
		\item $r_\bw(\kk) = (k_1'',\dots,k_s'') \in \Z^s$ where 
			\[
				k_i'' = k_i+1-\alpha_i-\beta_{i-1}, \qquad 1\le i\le s,
			\]
		\item $\tilde \bw = (\alpha_1\beta_1,\dots, \alpha_{s-1}\beta_{s-1})$, and 
		\item $|\bw| = \sum_{i=1}^{s-1}\alpha_i\beta_i$.
	\end{itemize}
\end{notation}

To help the readers understand the meaning of $W_{s-1}$, $d_\bw(\kk)$, and $r_\bw(\kk)$, we briefly give an additional explanation as follows. 
Let $\kk = (k_1, \cdots, k_s) \in (\Z_{\geq 1})^s$ so that the set of terminal vertices of $\Gamma_\kk$ is given by 
\[
	T_\kk = \{v_i = (n_i, n - n_i) \in V(\Gamma_\kk) ~|~ n_0 = 0, n_i = \sum_{j=1}^i k_j, i=1,\cdots,s\}.
\]
For a face $\gamma \in \Gamma_\kk$, the shape of $\gamma$ near a vertex $v_i \in T_\kk$ for $i \neq 0, s$ is one of three types : \\
\[	
\begingroup%
  \makeatletter%
  \providecommand\color[2][]{%
    \errmessage{(Inkscape) Color is used for the text in Inkscape, but the package 'color.sty' is not loaded}%
    \renewcommand\color[2][]{}%
  }%
  \providecommand\transparent[1]{%
    \errmessage{(Inkscape) Transparency is used (non-zero) for the text in Inkscape, but the package 'transparent.sty' is not loaded}%
    \renewcommand\transparent[1]{}%
  }%
  \providecommand\rotatebox[2]{#2}%
  \ifx\svgwidth\undefined%
    \setlength{\unitlength}{96.3796875bp}%
    \ifx\svgscale\undefined%
      \relax%
    \else%
      \setlength{\unitlength}{\unitlength * \real{\svgscale}}%
    \fi%
  \else%
    \setlength{\unitlength}{\svgwidth}%
  \fi%
  \global\let\svgwidth\undefined%
  \global\let\svgscale\undefined%
  \makeatother%
  \begin{picture}(1,0.2519011)%
    \put(0,0){\includegraphics[width=\unitlength,page=1]{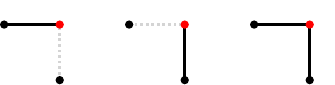}}%
    \put(0.92550622,0.21996361){\color[rgb]{0,0,0}\makebox(0,0)[lb]{\smash{$v$}}}%
    \put(0.17846084,0.21996336){\color[rgb]{0,0,0}\makebox(0,0)[lb]{\smash{$v$}}}%
    \put(0.55198353,0.21996336){\color[rgb]{0,0,0}\makebox(0,0)[lb]{\smash{$v$}}}%
  \end{picture}%
\endgroup%

\]	
Near $v_0$ and $v_s$, the shape of $\gamma$ is equal to 
\[
\begingroup%
  \makeatletter%
  \providecommand\color[2][]{%
    \errmessage{(Inkscape) Color is used for the text in Inkscape, but the package 'color.sty' is not loaded}%
    \renewcommand\color[2][]{}%
  }%
  \providecommand\transparent[1]{%
    \errmessage{(Inkscape) Transparency is used (non-zero) for the text in Inkscape, but the package 'transparent.sty' is not loaded}%
    \renewcommand\transparent[1]{}%
  }%
  \providecommand\rotatebox[2]{#2}%
  \ifx\svgwidth\undefined%
    \setlength{\unitlength}{73.0125bp}%
    \ifx\svgscale\undefined%
      \relax%
    \else%
      \setlength{\unitlength}{\unitlength * \real{\svgscale}}%
    \fi%
  \else%
    \setlength{\unitlength}{\svgwidth}%
  \fi%
  \global\let\svgwidth\undefined%
  \global\let\svgscale\undefined%
  \makeatother%
  \begin{picture}(1,0.33252012)%
    \put(0,0){\includegraphics[width=\unitlength,page=1]{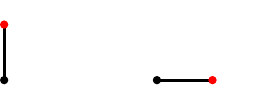}}%
    \put(0.83821263,0.07122135){\color[rgb]{0,0,0}\makebox(0,0)[lb]{\smash{$v_s$}}}%
    \put(0.01643554,0.29036124){\color[rgb]{0,0,0}\makebox(0,0)[lb]{\smash{$v_0$}}}%
  \end{picture}%
\endgroup%

\]
Thus the shape of $\gamma$ near $T_\kk$ is determined by the following map 
 \[
 		\mathcal{A}_\gamma : T_\kk \rightarrow \{ \to, \uparrow, \ur \}, \quad \mathcal{A}_\gamma(v_0) = ~\uparrow,  \quad \mathcal{A}_\gamma(v_s) = ~\to, 
 \]
 called an {\em assignment on $T_\kk$}, which is defined in the obvious way, see Figure \ref{figure_9_neighborhoods}. 
Then we may identify $W_{s-1}$ with the set of all assignments on $T_\kk$, 
where $\to$ corresponds to $(1,0)$, $\uparrow$ corresponds to $(0,1)$, and $\ur$ corresponds to $(1,1)$. 
Then $\widetilde{\bw}$ is the vector which assigns the position of $\ur$'s, and $|\bw|$ is the number of $\ur$'s in $\bw$ for each $\bw \in W_{s-1}$.  

Now, let us think of the geometric meaning of $r_\bw(\kk)$ and $d_\bw(\kk)$. 
For each $\bw = ((\alpha_1, \beta_1), \cdots, (\alpha_{s-1}, \beta_{s-1})) \in W_{s-1}$, let us consider a subgraph $g_\bw$ of $\Gamma_\kk$
such that the edge set of $g_\bw$ is defined to be 
\[
	E(g_\bw) := \{ ((v_i - (1,0), v_i)) ~|~ \alpha_i = 1\} \cup \{ (v_i - (0,1), v_i) ~|~ \beta_i = 1\}, 
\]
and the vertex set $V(g_\bw)$ is defined to be the set of endpoints of edges in $E(g_\bw)$. 

It is easy to check that $V(g_\bw) = V_{n-1}(g_\bw) \cup T_\kk$, where $V_{n-1}(g_\bw)$ is the set of vertices of $g_\bw$ 
lying on the line whose equation is given by $x+y = n-1$ where $n = \sum_{i=1}^s k_i$. See Figure \ref{figure_9_neighborhoods} for example : for each 
assignment $\bw$ on $T_\kk$, the blue dots are the vertices in $V_{n-1}(g_\bw)$, the red dots are the vertices in $T_\kk$, 
and the black line segments are edges of $g_\bw$.  
Thus $V(g_\bw)$ defines a unique 
ladder diagram, denoted by $\Gamma_\kk(\bw)$, whose set of terminal vertices is equal to $V_{n-1}(g_\bw)$. Then the following lemma interprets the geometric meaning of $r_\bw(\kk)$. 

\begin{lemma}\label{lemma_next_terminal_vertices}
	$\Gamma_{r_\bw(\kk) * \widetilde{\bw}} = \Gamma_\kk(\bw)$. 
\end{lemma}

\begin{proof}
	Note that each sequence $\kk = (k_1, \cdots, k_s) \in \Z^s_{\geq 0}$ defines the unique ladder diagram $\Gamma_\kk$.
	In particular, each $k_i$ is the 
	same as the difference of $y$-coordinates between two consecutive vertices $v_{i-1}$ and $v_i$ in $T_\kk$. 
	
	As seen in Figure \ref{figure_oprat}, each component $k_i''$ of $r_\bw(\kk)$ measures the difference of the $y$-coordinates of two 
	consecutive vertices in $V_{n-1}(g_\bw)$ which corresponds to $v_{i-1}$ and $v_i$ in $T_\kk$. 
	In fact, we can easily see that $k_i''$ is determined by the values $\alpha_i$ and $\beta_{i-1}$ and is equal to 
	\[
	k_i'' = k_i + 1 - \alpha_i - \beta_{i-1}
	\] in any case, see Figure \ref{figure_oprat}.(a). 
	 Also, observe that each $v_i$ with $\alpha_i \beta_i = 1$ produces two vertices in $V_{n-1}(g_\bw)$ 
	 such that the difference of their $y$-coordinates is $1$, see 
	 Figure \ref{figure_oprat}.(b). Thus the proof is straightforward. 
\end{proof}

\begin{figure}
	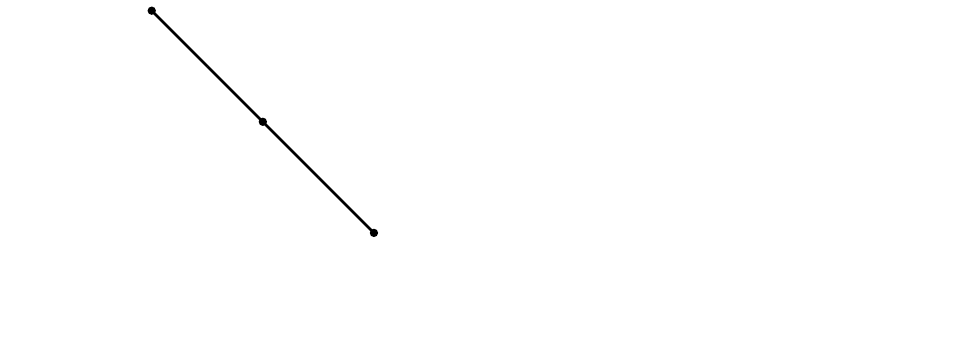
	\caption{\label{figure_oprat} Definition of $r_\bw(\kk)$ and $\widetilde{\bw}$.}
\end{figure}

Finally, the meaning of $d_\bw(\kk)$ is given as follows. 
 
\begin{lemma}\label{lem4}
For any $\kk \in \Z^s$, we have 
\[
	r_\bw(d_\bw(\kk) + \mathbf{1}) = \kk
\]
where $\mathbf{1} = (1,\cdots, 1) \in \Z^s_{\geq 1}$. 
In particular, if $d_\bw(\kk) \in \Z^s_{\ge 0}$, then so is $\kk$. 
\end{lemma}
\begin{proof}
The proof is straightforward from the definitions of $r_\bw(\kk)$ and $d_\bw(\kk)$. 
\end{proof}
 
 \begin{figure}
		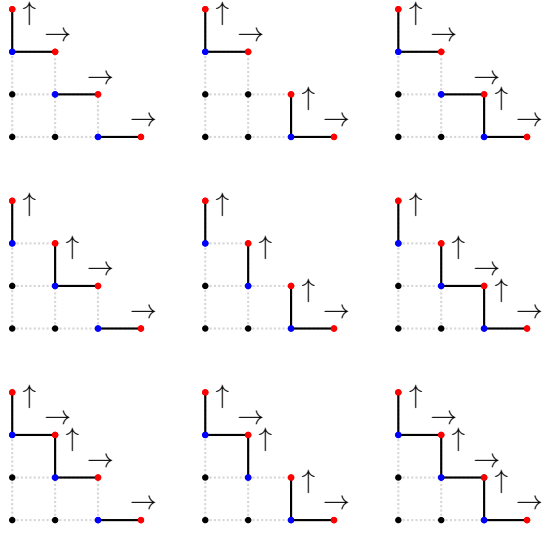
		\caption{\label{figure_9_neighborhoods} $3^{3-1}$ possible types of edges near $T_{(1,1,1)}$.}
\end{figure}

\subsection{Partial differential operators}
For each $\bw=((\alpha_1,\beta_1),\dots,(\alpha_{s-1},\beta_{s-1}))\in W_{s-1}$, let us denote by 
 \[
			D_\bw = \prod_{i=1}^{s-1} \left(\frac{\partial}{\partial x_i} \right)^{1-\alpha_i} \left(\frac{\partial}{\partial x_{i+1}} \right)^{1-\beta_i}
			\left(t\cdot \frac{\partial}{\partial y_i} \right)^{\alpha_i\beta_i}
\]
the differential operator defined on the ring of formal power series $\Q[[x_1, y_1, \cdots, x_{s-1}, y_{s-1}, x_s ; t]]$. 


\begin{lemma}\label{lem2}
The following identity holds : 
\[
\prod_{i=1}^{s-1}\left (\frac{\partial}{\partial x_i} + \frac{\partial}{\partial x_{i+1}}
		+ t \cdot \frac{\partial}{\partial y_i}\right )
= \sum_{\bw\in W_{s-1}} D_\bw.
\]  
\end{lemma}
\begin{proof}
Let
$\bw=((\alpha_1,\beta_1),\dots,(\alpha_{s-1},\beta_{s-1}))\in
W_{s-1}$. Then in each $i$th factor of the left hand side, we take 
$\frac{\partial}{\partial x_i}$ if $(\alpha_i,\beta_i)=(0,1)$, 
$\frac{\partial}{\partial x_{i+1}}$ if $(\alpha_i,\beta_i)=(1,0)$,
and $t\cdot \frac{\partial}{\partial y_i}$ if $(\alpha_i,\beta_i)=(1,1)$. 
Multiplying the chosen factors give $D_\bw$. Thus the expansion of the
left hand side is equal to the right hand side. 
\end{proof}

\begin{lemma}\label{lem3}
Let $\kk\in \Z^s_{\ge 0}$, $\ee\in \Z^{s-1}_{\ge 0}$
and $\bw\in W_{s-1}$. Then
\[
\left.\left( D_\bw\left(
\frac{(\xx * \yy)^{\kk*\ee}}{(\kk*\ee)!} \right) \right)\right|_{\yy =
0} =  \left\{
\begin{array}{ll}
t^{|\bw|}\cdot \frac{\xx^{d_\kk(\bw)}}{d_\kk(\bw)!}, 
& \mbox{if $\ee=\tilde\bw$ and $d_\bw(\kk)\in \Z^s_{\ge 0}$,}\\
  0, & \mbox{otherwise.}
\end{array}
\right.
\]
\end{lemma}
\begin{proof}
Note that the order of $\frac{\partial}{\partial x_i}$ in $D_\bw$ is $2-\alpha_i - \beta_{i-1}$, and the order of 
$\frac{\partial}{\partial y_i}$ is $\alpha_i \beta_i$ for each $i=1,\cdots, s$.  
Thus 
\[
	\begin{array}{ccl}
	\displaystyle D_\bw\left( \frac{(\xx * \yy)^{\kk*\ee}}{(\kk*\ee)!}\right)\Big|_{\yy = 0} = 0 &  \Leftrightarrow &  e_i \neq \alpha_i \beta_i \quad \text{or}
	\quad k_i < 2 - \alpha_i - \beta_{i-1} ~\text{for some} ~1 \leq i \leq s\\
	& \Leftrightarrow &  e \neq \widetilde{\bw} \quad \text{or} \quad d_\bw(\kk) \not \in \Z^s_{\ge 0}. \\
	\end{array}
\]  
For the other case, we have 
\begin{eqnarray*}
	 D_\bw\left( \frac{(\xx * \yy)^{\kk*\ee}}{(\kk*\ee)!} \right) & = & D_\bw \left( \prod_{i=1}^s \frac{x_i^{k_i}}{k_i!} \prod_{j=1}^{s-1} \frac{y_j^{e_j}}{e_j!} \right)\\
	                                                                                           & = & \prod_{i=1}^s \frac{x_i^{k_i-2+\alpha_i + \beta_{i-1}}}{(k_i-2+\alpha_i + \beta_{i-1})!} \prod_{j=1}^{s-1} 
	                                                                                           		\frac{y_j^{e_j - \alpha_j \beta_j}}{(e_j - \alpha_j \beta_j)!} \cdot t^{\alpha_j \beta_j}\\[0.5em]
	                                                                                           & = & t^{|\bw|}\cdot \frac{\xx^{d_\bw(\kk)}}{d_\bw(\kk)!}.
\end{eqnarray*}
This completes the proof. 
\end{proof}

%
\subsection{Proof of the main theorem} We start with the following lemma.

\begin{lemma}\label{lem1}
For $\kk\in\Z^s_{\ge 1}$, we have
\begin{equation}\label{equation_recursion}
F_\kk(t) = \sum_{\bw\in W_{s-1}} F_{r_\bw(\kk) * \tilde\bw}(t) \cdot t^{|\bw|}.
\end{equation}
\end{lemma}
\begin{proof}
Note that the left hand side of (\ref{equation_recursion}) is equal to 
\[
F_\kk(t) = \sum_{\gamma\in \mathcal{F}(\Gamma_\kk)} t^{\dim \gamma}
\]
by definition, and the right hand side of (\ref{equation_recursion}) is equal to 
\[
	\sum_{\bw\in W_{s-1}} F_{r_\bw(\kk) * \tilde\bw}(t) \cdot t^{|\bw|} = 
	\sum_{\substack{\bw\in W_{s-1} \\ \sigma \in \F(\Gamma_{r_\bw(\kk) * \widetilde{\bw}})}} t^{\dim \sigma} \cdot t^{|\bw|}. 
\]	
Let
\[
X = \{(\bw, \sigma): \bw\in W_{s-1}, 
\sigma \in \mathcal{F}(\Gamma_{r_\bw(\kk) * \tilde\bw}) \}.
\]
Then it is sufficient to find a bijection 
\[
\phi : X \to \mathcal{F}(\Gamma_\kk) 
\]
such that $\dim \phi(\bw,\sigma) = |\bw| + \dim \sigma$. 

Recall that the set of terminal vertices of $\sigma \in \F(\Gamma_{r_\bw(\kk) * \widetilde{\bw}})$ is equal to $V_{n-1}(g_\bw)$ where 
\[
	V_{n-1}(g_\bw) = V(g_\bw) \cap \{ (x,y) ~|~ x+y = n-1\}. 
\]
Thus we can define $\phi$ to be 
\[
	\phi(\bw, \sigma) := \sigma \cup g_\bw
\]
which is clearly a face of $\Gamma_\kk$. 

Conversely, any face $\gamma \in \F(\Gamma_\kk)$ contains $g_\bw$ where $\bw$ corresponds to the assignment $\mathcal{A}_\gamma$ and it can be decomposed 
into 
\[
	\gamma = \sigma \cup g_\bw
\]
where $\sigma \in \F(\Gamma_{r_\bw(\kk) * \widetilde{\bw}})$ is a full subgraph of $\gamma$ obtained from removing terminal vertices $T_\kk$ of $\gamma$.
Then it defines a map $\psi :  \mathcal{F}(\Gamma_\kk) \to X$ such that $\psi(\gamma) := (\sigma, \bw)$. It is clear that $\psi \circ \phi$ is the identity map 
on $X$.

Finally for every $(\bw, \sigma) \in X$, each $v_i$ with $(\alpha_i, \beta_i) = (1,1)$ generates exactly one cycle in $\phi(\bw, \sigma)$ containing $v_i$. 
Thus we have $\dim \phi(\bw, \sigma) = |\bw| + \dim \sigma$.  

\end{proof}

Note that 
		\[
\displaystyle\frac{\partial^s}{\partial x_1 \cdots \partial x_s} \Psi_s(\xx;t)  =
\displaystyle \left.\frac{\partial^s}{\partial x_1 \cdots \partial x_s} \Psi_{2s-1}(\xx * \yy ; t) \right|_{\yy=0}.\\
\]
Thus the following theorem is equivalent to our main theorem \ref{theorem_global_generating_function}. 

\begin{theorem}[Theorem \ref{theorem_global_generating_function}] Under the same assumption, 
	\[
\frac{\partial^s}{\partial x_1 \cdots \partial x_s} \Psi_s(\xx;t)
=\left.\left(\prod_{i=1}^{s-1}\left (\frac{\partial}{\partial x_i} + \frac{\partial}{\partial x_{i+1}}
		+ t \cdot \frac{\partial}{\partial y_i}\right )
                \right)\right|_{\yy = 0} \Psi_{2s-1}(\xx * \yy;t).
	\]
\end{theorem}

\begin{proof}
By Lemma~\ref{lem1}, the left hand side is
\begin{equation}
  \label{eq:1}
\sum_{\kk\in \Z^s_{\ge 1}} F_\kk(t) \frac{\xx^{\kk-\mathbf1}}{(\kk-\mathbf1)!}
=\sum_{\kk\in \Z^s_{\ge 0}} F_{\kk+\mathbf1}(t) \frac{\xx^{\kk}}{\kk!}
=\sum_{\kk\in \Z^s_{\ge 0}}
\left(\sum_{\bw\in W_{s-1}}F_{r_\bw(\kk+\mathbf1) * \tilde\bw}(t) \cdot t^{|\bw|}\right)
\frac{\xx^{\kk}}{\kk!}.
\end{equation}
Observe that
\[
\Psi_{2s-1}(\xx * \yy;t) = 
\sum_{\substack{\kk\in \Z^s_{\ge 0}\\
\ee\in \Z^{s-1}_{\ge 0}}} 
F_{\kk*\ee}(t) \frac{(\xx * \yy)^{\kk*\ee}}{(\kk*\ee)!}.
\]
By Lemma~\ref{lem2} and the above identity, the right hand side of the theorem is
\[
\left.\left(\sum_{\bw\in W_{s-1}} D_\bw
\left(\Psi_{2s-1}(\xx * \yy)\right)  \right)\right|_{\yy = 0}
=\sum_{\bw\in W_{s-1}} 
\sum_{\substack{\kk\in \Z^s_{\ge 0}\\
\ee\in \Z^{s-1}_{\ge 0}}}  F_{\kk*\ee}(t) 
\left( D_\bw
\left.\left(\frac{(\xx * \yy)^{\kk*\ee}}{(\kk*\ee)!}
\right)\right)\right|_{\yy = 0}.
\]
By Lemma~\ref{lem3}, this is equal to
\[
\sum_{\bw\in W_{s-1}} 
\sum_{\substack{\kk\in \Z^s_{\ge 0}\\
d_\bw(\kk)\in \Z^s_{\ge 0}}}  F_{\kk*\tilde\bw}(t)  \cdot
t^{|\bw|}\cdot \frac{\xx^{d_\bw(\kk)}}{d_\bw(\kk)!}.
\]
Note that $d_\bw(\kk) \in \Z^s_{\geq 0}$ implies $\kk  = r_\bw(d_\bw(\kk) + \mathbf{1}) \in Z^s_{\geq 0}$ by Lemma by Lemma~\ref{lem4}.
Thus by letting $\kk' = d_\bw(\kk)$, the above sum becomes
\[
\sum_{\bw\in W_{s-1}} 
\sum_{\kk'\in \Z^s_{\ge 0}}  F_{r_\bw(\kk'+1)*\tilde\bw}(t)  \cdot
t^{|\bw|}\cdot \frac{\xx^{\kk'}}{(\kk')!},
\]
which is equal to \eqref{eq:1}. This completes the proof. 
\end{proof}

\bibliographystyle{annotation}

\end{document}